\documentclass[11pt]{amsart}
\usepackage{amsmath, amssymb, amsthm,fullpage, epic, eepic}
\usepackage[all]{xypic}

    \def\ps@copyright{\ps@empty
    \def\@oddfoot{\hfil\small\copyright 1997, \SMF}}
\makeatother

\newcommand{\BibTeX}{{\scshape Bib}\kern-.08em\TeX}
\newcommand{\T}{\S\kern .15em\relax }
\newcommand{\AMS}{$\mathcal{A}$\kern-.1667em\lower.5ex\hbox
        {$\mathcal{M}$}\kern-.125em$\mathcal{S}$}

\newtheorem{definition}{Definition}[section]
\newtheorem{theorem}[definition]{Theorem}
\newtheorem{proposition}[definition]{Proposition}

\newtheorem{lemma}[definition]{Lemma}
\newtheorem{remark}[definition]{Remark}

\newcommand{\R}{\mathbb{R}}
\newcommand{\Z}{\mathbb{Z}}

\newcommand{\noadd}[1]{}
\def\R{\mathbb{R}}

\input xy

\xyoption{all}

\newcommand{\Conv}{\operatorname{Conv}}
\newcommand{\Cent}{\operatorname{Cent}}

\title{On a conjecture of Kottwitz and Rapoport}
\title[On a conjecture of Kottwitz and Rapoport]{On a conjecture of Kottwitz and Rapoport}
\author{Q\"endrim R. Gashi}

\address{Max-Planck-Institut f\"ur Mathematik\\
Vivatsgasse 7, 53111 Bonn}
\email{qendrim@mpim-bonn.mpg.de}

\begin{document}
\def\smfbyname{}

\begin{abstract}
We prove a conjecture of Kottwitz and Rapoport which implies a converse to Mazur's Inequality for all split and quasi-split (connected) reductive groups. These results are related to the non-emptiness of certain affine Deligne-Lusztig varieties.
\end{abstract}

\maketitle


\section{Introduction}

Mazur's Inequality (\cite{mazur}, \cite{mazur2}) is related to the study of $p$-adic estimates of the number of points of certain algebraic varieties over a finite filed of characteristic $p$. But, it is most easily stated using isocrystals, where an isocrystal is a pair $(V,\Phi)$, with $V$ being a finite-dimensional vector space over the fraction field, $K$, of the ring of Witt vectors $W(\overline{ \mathbb{F}}_p)$, equipped with a $\sigma$-linear bijective endomorphism $\Phi$ of $V$, where $\sigma$ is the automorphism of $K$ induced by the Frobenius automorphism of $\overline{\mathbb{F}}_p$. We now recall Mazur's inequality.

Suppose that $(V,\Phi)$ is an isocrystal of dimension $n$. By Dieudonn\'{e}-Manin theory, we can associate to $V$ its Newton vector $\nu (V,\Phi) \in (\mathbb{Q}^n)_+ := \{ (\nu_1,\ldots,\nu_n) \in \mathbb{Q}^n: \nu_1 \geq \nu_2 \geq \ldots \geq \nu_n \}$, which classifies isocrystals of dimension $n$ up to isomorphism. If $\Lambda$ is a $W(\overline{ \mathbb{F}}_p)$-lattice in $V$, then we can associate to it the Hodge vector $\mu (\Lambda) \in (\mathbb{Z}^n)_+ : = (\mathbb{Q}^n)_+ \cap \mathbb{Z}^n$, which measures the relative position of the lattices $\Lambda$ and $\Phi(\Lambda)$. Denote by $\geq$ the usual dominance order. Mazur's Inequality asserts that $\mu (\Lambda) \geq \nu (V,\Phi)$.

A converse to this inequality was proved by Kottwitz and Rapoport in \cite{krapo}, where they showed that if we let $(V,\Phi)$ be an isocrystal of dimension $n$, and let $\mu=(\mu_1,...,\mu_n) \in (\mathbb{Z}^n)_+$ be such that $\mu \geq \nu(V,\Phi)$, then there exists a $W(\overline{ \mathbb{F}}_p)$-lattice $\Lambda$ in $V$ satisfying $\mu = \mu(\Lambda)$.

Both Mazur's Inequality and its converse can be regarded as statements for the group $GL_n$ since the dominance order arises naturally in the context of the root system for $GL_n$. In fact, there is a bijection (see \cite{kotiso}) between isomorphism classes of isocrystals of dimension $n$ and the set of $\sigma$-conjugacy classes in $GL_n(K)$. Kottwitz studies in ibid. the set $B(G)$ of the $\sigma$-conjugacy classes in $G(K)$, for a connected reductive group $G$ over $\mathbb{Q}_p$, and, as he notes, there is a bijection between $B(G)$ and the isomorphism classes of isocrystals of dimension $n$ with ``$G$-structure'' (for $G=GL_n$ these are simply isocrystals). So, results about isocrystals, and more generally isocrystals with additional structure, are related to results about the $\sigma$-conjugacy classes of certain reductive groups.

With this viewpoint in mind, we are interested in the group-theoretic generalizations of Mazur's Inequality and its converse, especially since they appear naturally in the study of the non-emptiness of certain affine Deligne-Lusztig varieties. To make these statements more precise,  we introduce some notation.

Let $F$ be a finite extension of $\mathbb{Q}_p$, with uniformizing element $\pi$. Denote by $\mathfrak{o}_F$ the ring of integers of $F$. Suppose $G$ is a split connected reductive group, $B$ a Borel subgroup and $T$ a maximal torus in $B$, all defined over $\mathfrak{o}_F$. (Quasi-split groups are treated in the last section of the paper.) Let $L$ be the completion of the maximal unramified extension of $F$ in some algebraic closure of $F$, let $\sigma$ be the Frobenius of $L$ over $F$, and let $\mathfrak{o}_L$ be the valuation ring of $L$.

We write $X$ for the set of co-characters $X_*(T)$. Let $\mu \in X$ be a dominant element and $b\in G(L)$. The affine Deligne-Lusztig variety $X_{\mu}^G(b)$ is defined by $$X_{\mu}^G(b):= \{ x \in G(L)/G(\mathfrak{o}_L) : \, x^{-1} b \sigma (x) \in G(\mathfrak{o}_L) \mu (\pi) G(\mathfrak{o}_L) \} .$$ These $p$-adic ``counterparts'' of the classical Deligne-Lusztig varieties, get their name by virtue of being defined in the same way as the latter and have been studied by a number of authors. See, for example, \cite{GHKR}, \cite{GHKR2}, \cite{eva}, and references therein. For the relevance of affine Deligne-Lusztig varieties to Shimura varieties, the reader may wish to consult \cite{rapoport2}.

We need some more notation to be able to formulate the group-theoretic generalizations of Mazur's Inequality and its converse. Let $P=MN$ be a parabolic subgroup of $G$ which contains $B$, where $M$ is the unique Levi subgroup of $P$ containing $T$. The Weyl group of $T$ in $G$ will be denoted by $W$. We let $X_{G}$ and $X_{M}$ be the quotient of $X$ by the coroot lattice for $G$ and $M$, respectively. Also, we let $\varphi_{G}:X\rightarrow X_G$ and $\varphi_{M}:X\rightarrow X_M$ denote the respective natural projection maps.

Let $B=TU$, with $U$ the unipotent radical. If $g \in G(L)$, then there is a unique element of $X$, denoted $r_B(g)$, so that $g \in G(\mathfrak{o}_L) \, r_B(g)(\pi) U(L)$. If the image of $r_B(g)$ under the canonical surjection $X \rightarrow X_G$ is denoted by $w_G(g)$, then we have a well-defined map $w_G: G(L) \rightarrow X_G$, the Kottwitz map \cite{kotiso}. In a completely analogous way one defines the map $w_M: M(L) \rightarrow X_M$, where one considers $M$ instead of $G$.

We use the partial ordering $\stackrel{P}{\leq}$ in $X_M$, where for $\mu, \nu \in X_M$, we write $\nu \stackrel{P}{\leq} \mu$ if and only if $\mu - \nu$ is a nonnegative integral linear combination of the images in $X_M$ of the coroots corresponding to the simple roots of $T$ in $N$.

We can now state the first main result in this paper.

\begin{theorem}
Let $\mu \in X$ be dominant and let $b \in M(L)$ be a basic element such that $w_M(b)$ lies in $X_M^{+}$. Then
$$X_{\mu}^G(b) \neq \emptyset \Longleftrightarrow  w_M(b) \stackrel{P}{\leq} \mu,$$ where in the last relation we consider $\mu$ as an element of $X_M$. \emph{(}For the definition of basic see \cite{kotiso} and for the definition of $X_M^{+}$ see \cite{kot}.\emph{)}
\label{groupmazur}
\end{theorem}

We also prove a similar theorem for unramified groups. The precise formulation (Theorem \ref{mainquasi}) and the proof of that result is postponed to the last section of the paper.

One direction in Theorem \ref{groupmazur}, namely $$X_{\mu}^G(b) \neq \emptyset \Longrightarrow w_M(b) \stackrel{P}{\leq} \mu,$$ is the group-theoretic generalization of Mazur's Inequality, and it is proved by Rapoport and Richartz in \cite{rapritch} (see also \cite{kot}, Theorem 1.1, part (1)).

The other direction, i.e., the group-theoretic generalization of the converse to Mazur's Inequality, is a conjecture of Kottwitz and Rapoport \cite{krapo}. Next, we discuss how their conjecture is reduced to one formulated only in terms of root systems.

Let $$\mathcal{P} _{\mu}: =\left\{\nu\in X: (i)\, \varphi_{G}(\nu)=\varphi_{G}(\mu); \, \emph{and} \,\,(ii)\, \nu\in \Conv\left(W\mu\right)\right\},$$ where $\Conv\left(W\mu\right)$ is the convex hull of $W\mu :=\left\{w(\mu): w\in W\right\}$ in $\mathfrak{a} :=X\otimes_{\mathbb{Z}}\mathbb{R}$. Then we have (cf. \cite{kot}, Theorem 4.3) $$X_{\mu}^G(b) \neq \emptyset \Longleftrightarrow w_M(b) \in \varphi_M (\mathcal{P}_{\mu}).$$

This way we see that the other implication in Theorem \ref{groupmazur} follows if we show that $$w_M(b) \in \varphi_M (\mathcal{P}_{\mu}) \Longrightarrow  w_M(b) \stackrel{P}{\leq} \mu,$$ i.e., if we show that for $\nu \in X_M^+$ we have
\begin{equation}\label{bllah}
\nu \in \varphi_M (\mathcal{P}_{\mu}) \Longrightarrow  \nu \stackrel{P}{\leq} \mu.
\end{equation}

This is guaranteed by the following:

\begin{theorem}\label{conjecture} \emph{(Kottwitz-Rapoport Conjecture; split case)} We have that
$$\varphi_M\left(\mathcal{P}_\mu\right)=\left\{\nu\in X_M :(i)\,\nu, \mu \,\,\emph{have the same image in}\,\, X_G;\right.$$
$$\left. \hspace{6.0cm} (ii)\, \emph{the image of}\, \,\nu\, \emph{in} \,\mathfrak{a}_M \,\emph{lies in}\, pr_M\left(\Conv \left( W\mu \right) \right)\right\},$$ where $\mathfrak{a}_M:=X_M\otimes_{\Z}\R$ and $\emph{pr}_M:\mathfrak{a}\rightarrow\mathfrak{a}_M$ denotes the natural projection induced by $\varphi_M$.
\end{theorem}

To see that the right-hand side of (\ref{bllah}) corresponds to the right-hand side of Theorem \ref{conjecture}, we refer the reader to Section 4.4 of \cite{kot}.

A variant of Theorem \ref{conjecture}, in the case of quasi-split groups, is proved in the last section (see Theorem \ref{conjecture2}). We remark that Theorem \ref{conjecture} is a statement that is purely a root-theoretic one, so it remains true when we work over other fields of characteristic zero, not just $\mathbb{Q}_p$.

Theorem \ref{conjecture} had been previously proved for $GL_n$ and $GSp_{2n}$ by Kottwitz and Rapoport \cite{krapo} and then for all classical groups by Lucarelli \cite{cathy}. In addition, Wintenberger, using different methods, proved this result for $\mu$ minuscule (see \cite{win}). A more general version of this theorem for $GL_n$ was proved in \cite{qendrim} (Theorem A in loc. cit.) using the theory of toric varieties.

We prove Theorem \ref{conjecture} in the rest of the paper, but we next recall an interesting relation between Theorem \ref{conjecture} and cohomology-vanishing on toric varieties associated with root systems (for more details see \cite{qendrim}, \cite{qendrim2}). Let $\hat{G}$ and $\hat{T}$ be the (Langlands') complex dual group for $G$ and $T$, respectively. Let $Z(\hat{G})$ be the center of $\hat{G}$. Let $V_G$ be the (projective nonsingular) toric variety whose fan is the Weyl fan in $X_* (\hat{T}/Z(\hat{G})) \otimes_{\Z}\R$ and whose torus is $\hat{T}/Z(\hat{G})$. We are interested in the action of $\hat{T}$ on $V_G$, which is obtained using the canonical surjection $\hat{T} \twoheadrightarrow \hat{T}/Z(\hat{G})$ and the action of $\hat{T}/Z(\hat{G})$ on $V_G$.

The theory of toric varieties is famous for its rich dictionary between combinatorial convexity and algebraic geometry. For the toric varieties $V_G$ one also gets group-theoretic information in the picture. For example, we have a one-to-one correspondence between Borel subgroups of $G$ containing $T$ and $\hat{T}$-fixed points in $V_G$.  Of special interest for us are certain globally generated line bundles on $V_G$ which arise from Weyl orbits (for the precise definition see \cite{qendrim}). Recall that $M$ is a Levi subgroup containing $T$. We then have a toric variety $Y_M^G$ for the torus $Z(\hat{M})/Z(\hat{G})$, whose definition we only recall for $\hat{G}$ adjoint (for general $G$ see \cite{akot}, \S23.2). In this case $Z(\hat{M})$ is a subtorus of $\hat{T}$ and so $X_*(Z(\hat{M}))$ is a subgroup of $X_*(\hat{T})$. The collection of cones from the Weyl fan inside $X_*(\hat{T})\otimes_{\mathbb{Z}}\mathbb{R}$ that lie in the subspace $X_*(Z(\hat{M})) \otimes_{\mathbb{Z}}\mathbb{R}$ gives a fan. This is the fan for the complete, nonsingular, projective toric variety $Y_M^G$.

For brevity of our exposition, let us now assume that our parabolic subgroup $P$ is of semisimple rank 1. This implies that the root lattice $R_{\hat{M}}$ is just $\mathbb{Z} \alpha$ for a unique, up to a sign, root $\alpha$ of $\hat{G}$, and that the toric variety $Y_M^G$, which we now denote by $D_{\alpha}$, is a (non-torus-invariant) divisor in $V_G$. The map $\varphi_M$ will now be denoted by $p_{\alpha}$. Passing to the complex dual world and tensoring with $\mathbb{R}$ we get a map which we still denote by $p_{\alpha}:X^*(\hat{T}) \otimes_{\mathbb{Z}}\mathbb{R} \twoheadrightarrow ( X^*(\hat{T})/\mathbb{Z}\alpha )\otimes_{\mathbb{Z}}\mathbb{R} .$

Let $\mathcal{L}$ be a $\hat{T}$- line bundle on $V_G$ that is generated by its sections. Then we have a short exact sequence of sheaves on $V_G$: $$0\longrightarrow \mathcal{J}_{D_{\alpha}} \otimes \mathcal{L} \longrightarrow \mathcal{L} \longrightarrow i_*(\mathcal{L}|_{D_{\alpha}}) \longrightarrow 0,$$
where $\mathcal{J}_{D_{\alpha}}$ is the ideal sheaf of $D_{\alpha}$ and $i$ is the inclusion map $D_{\alpha} \hookrightarrow V_G$. Note that $$H^i(V_G, \mathcal{L})=0,\,\, H^i(V_G,i_*(\mathcal{L}|_{D_{\alpha}}))=H^i(D_{\alpha},\mathcal{L}|_{D_{\alpha}})=0,$$ for all $i>0$, since $\mathcal{L}$ and $\mathcal{L}|_{D_{\alpha}}$ are generated by their sections and $V_G$ and $D_{\alpha}$ are projective toric varieties. Therefore the above short exact sequence gives rise to the long exact sequence
$$\,\,\,\,\,... \longrightarrow H^0(V_G, \mathcal{L}) \stackrel{\varphi}{\longrightarrow} H^0(V_G,i_*(\mathcal{L}|_{D_{\alpha}})) \longrightarrow H^1(V_G,\mathcal{J}_{D_{\alpha}} \otimes \mathcal{L}) \longrightarrow 0.$$ Thus, the surjectivity of the map $\varphi$ is equivalent to $H^1(V_G,\mathcal{J}_{D_{\alpha}} \otimes \mathcal{L})=0.$

The surjectivity of $\varphi$ follows from Theorem \ref{conjecture}, for certain line bundles $\mathcal{L}$.

\begin{theorem} With notation as above, we have that $$H^i(V_{G},\mathcal{J}_{D_{\alpha}} \otimes \mathcal{L})=0, \forall i \geq 1,$$ whenever $\mathcal{L}$ is a globally generated line bundle arising from a Weyl orbit.
\end{theorem}

Let us just mention that the proof of this theorem uses, among other facts, a concrete description for the dimension of the space of global sections of a line bundle on a toric variety in terms of lattice points in certain polytopes (cf. \cite{fulton}, pg.66).

We note that in the case of the group $G=GL_n$, a stronger result than that of the previous theorem is true. In fact, in \cite{qendrim} it is proved that $$H^i(V_{GL_n},\mathcal{J}_{D_{\alpha}} \otimes \mathcal{L})=0,\forall i \geq 1$$ whenever $\mathcal{L}$ is globally generated.

In the end, let us describe how our paper is organized. In Section 2 we prove Theorem \ref{conjecture} in the simply-laced case. Some auxiliary results used in this proof are treated in the next section. An interesting feature of the proof of Theorem \ref{conjecture} is that its last part involves Peterson's notion of minuscule Weyl group elements (cf. \cite{stem}) or, equivalently, the numbers game with a cutoff \cite{GS} --- this is a modified version of the so-called Mozes' game of numbers (cf. \cite{mozes}). Section 4 is devoted to the proof of Theorem \ref{conjecture} in the non-simply laced case where we use a folding argument to deduce the result from the analogous statement for the simply-laced one. The last section contains the proof of a converse to Mazur's Inequality for quasi-split groups.

\vspace{0.2in}

{\bf Acknowledgments:} It is with special pleasure and great gratitude that we thank Robert Kottwitz for his time, invaluable advice and comments, and for carefully reading earlier versions of this paper. We heartily thank Travis Schedler for allowing the inclusion in this paper of joint results appearing on Section 3 and for very fruitful discussions on the numbers game. We also thank Michael Rapoport and Ulrich G\"ortz for their comments on an earlier version of this paper, and  Eva Viehmann for helpful conversations. We thank Artan Berisha for help with a computer program. Part of this work was supported by an EPDI Fellowship and a Clay Liftoff Fellowship. We thank the University of Chicago and the Max Planck Institute of Mathematics in Bonn for their hospitality.

\section{The case of simply-laced root systems}

Since the statement of Theorem \ref{conjecture} only involves root systems and since we will be using facts from \cite{bourbaki}, we shall rewrite the statement of our main result so that it conforms to the notation from \cite{bourbaki}. Moreover, we will be working with roots, instead of coroots (which can also be interpreted to mean that we will be working with the Langlands' complex dual group of $G$, instead of with the group $G$ itself).

Suppose that $R$ is a (reduced, irreducible) root system and $W$ is its Weyl group. Denote by $P(R)$ and $Q(R)$ the weight and radical-weight lattices for $R$, respectively. Let $\Delta := \{\alpha_i : i \in I \}$, where $I:= \{1,\ldots,n\}$, be the simple roots (for some choice) in $R$. Let $\emptyset \neq J \subsetneqq  I$ and consider the sub-root system, denoted $R_J$, corresponding to the set of roots $\{\alpha_j : j \in J \}$ (this corresponds to the Levi group $M$ from the Introduction). Let $Q(R_J)$ be defined similarly to $Q(R)$.

Let $\mu \in P(R)$ be a dominant weight, i.e., $\langle \mu , \alpha_i^{\vee} \rangle \geq 0 , \, \forall i \in I$, where $\alpha_i^{\vee}$ is the coroot corresponding to $\alpha_i$, and $\langle \, , \rangle$ stands for the canonical pairing between weights and coweights of $R$. Then consider the convex hull $\Conv (W\mu)$ inside $P(R)\otimes_{\mathbb{Z}}\mathbb{R}$. Let $\varphi$ and $\varphi_J$ be the canonical projections of $P(R)$ onto $P(R)/Q(R)$ and onto $P(R)/Q(R_J)$, respectively. Recall that we defined $$\mathcal{P}_{\mu}: = \left\{\nu\in P(R): (i)\, \varphi(\nu)=\varphi(\mu); \, \emph{and} \,\,(ii)\, \nu\in \Conv\left(W\mu\right)\right\}.$$ If we write $pr_J$ for the natural projection $$P(R)\otimes_{\mathbb{Z}}\mathbb{R} \rightarrow (P(R)/Q(R_J)) \otimes_{\mathbb{Z}}\mathbb{R},$$ induced by $\varphi_J$, then Theorem \ref{conjecture} can be reformulated as follows.

\begin{theorem}
We have that $$\varphi_J \left(\mathcal{P}_\mu\right)=\left\{y \in P(R)/Q(R_J) :(i)\,y, \mu \,\,\emph{have the same image in}\,\, P(R)/Q(R);\right.$$
$$\left. \hspace{2cm} (ii)\, \emph{the image of}\, \,y \, \emph{in} \,(P(R)/Q(R_J)) \otimes_{\mathbb{Z}}\mathbb{R} \,\emph{lies in}\, pr_J\left(\Conv \left( W\mu \right) \right)\right\}.$$

\label{thm}
\end{theorem}

Note that to prove Theorem \ref{thm} it is sufficient to prove that the right-hand side is contained in the left-hand side, since the converse is clear.

Suppose that $y$ is an element of the set appearing on the right-hand side in Theorem \ref{thm}. Without loss of generality, we can assume that $y$ is dominant. We then have a unique element $z \in P(R)$ which is $J$-minuscule, $J$-dominant and such that $pr_J(z)=y$ (cf. Proposition 8, \S7, Ch. VIII in \cite{bourbaki2}). We recall that $z$ being $J$-minuscule means that $\langle z , \alpha^{\vee} \rangle  \in \{ -1, 0, 1\},$ for all roots $\alpha$ in $R_J$, and $z$ being $J$-dominant means that $\langle z , \alpha_j^{\vee} \rangle \geq 0, \, \forall j \in J$. (If we do not modify the adjectives dominant and minuscule, then they will always mean $I$-dominant and $I$-minuscule.)

We can consider $(P(R)/Q(R_J)) \otimes_{\mathbb{Z}}\mathbb{R}$ as a subspace of $P(R) \otimes_{\mathbb{Z}}\mathbb{R}$, and then we can write $$z=y+ \sum_{j\in J} k_j \alpha_j,$$ for some non-negative reals $k_j$. Instead of $z$, consider $$z'=y+ \sum_{j\in J} k^{'}_j \alpha_j,$$ where, for each $j$, $k^{'}_j$ stands for the fractional part of $k_j$. Clearly, $pr_J(z')=y$. Then Theorem \ref{thm} follows from the following result.

\begin{proposition}
The element $z'$ lies in $\mathcal{P}_{\mu}.$
\label{prop}
\end{proposition}

Perhaps we should mention here that a similar proposition (for classical groups) was proved in \cite{cathy}, but there $z$ was shown to lie in $\mathcal{P}_{\mu}$ and $z'$ was not considered at all. For our proof, as will become apparent shortly, it is essential that we consider $z'$ instead of $z$. It turns out, however, that (at least in the simply-laced cases) $z$ and $z'$ are in the same Weyl orbit (Lemma \ref{jminlem}), and therefore the above proposition remains true when $z'$ is replaced by $z$.

Since we have assumed that $\mu$ and $y$ have the same image in $P(R)/Q(R)$, we immediately get that $\mu$ and $z'$ also have the same image in $P(R)/Q(R)$. Thus, to prove Proposition \ref{prop}, we only need to show that $z' \in \Conv\left(W\mu\right)$, which will indeed occupy the rest of the paper.

Before we start with some auxiliary results, let us make an important assumption. We will assume that $R$ is a simply-laced root system. The result of Theorem \ref{thm} for the non-simply laced root systems will follow from the analogous result for the simply-laced root systems by the well-known argument of folding. This is carried out in Section \ref{nonsimplylaced}.

One of the difficulties is that the element $z'$, like $z$, is not dominant in general. So, we let $w' \in W$ be such that $w'(z')$ is dominant. Then to show that $z' \in \Conv\left(W\mu\right)$, it suffices to prove that $\langle w'(z'), \omega_i \rangle \leq \langle \mu , \omega_i \rangle, \forall i \in I$, where, for all $i \in I$, $\omega_i$ stands for the fundamental coweight corresponding to $\alpha_i$. The strategy for the proof of these inequalities will be to construct an element $w'$ as above in such a way that we get the inequalities for free or with very little work.

Let us first introduce some more terminology. For $\lambda \in P(R)$ and $w \in W$, we say that $w$ is \emph{$\lambda$-minuscule} if there is a reduced expression $w=s_{i_1}s_{i_2}\cdots s_{i_t}$ such that $$s_{i_r}s_{i_{r+1}}\cdots s_{i_t} \lambda = \lambda + \alpha_{i_r} + \alpha_{i_{r+1}} + \ldots + \alpha_{i_t},\,\, 1 \leq r \leq t,$$ where for any $i$, $s_i \in W$ stands for the simple reflection corresponding to $\alpha_i$. It is easily seen that $w$ is $\lambda$-minuscule if and only if $$\langle s_{i_{r+1}}\cdots s_{i_t} \lambda , \alpha_{i_r}^{\vee} \rangle = -1, \,\, 1 \leq r \leq t.$$

Note that usually one defines $\lambda$-minuscule weights by requiring, in the last equalities, that the left-hand side equals to $+1$ as opposed to $-1$. For more on minuscule Weyl group elements, a notion invented by Peterson, see for example \cite{stem}. We point out that the notion of a Weyl group element $w$ being $\lambda$-minuscule does not depend on the choice of the reduced expression for $w$. This statement is proved in \cite{stem}, Proposition 2.1.

The next result reveals what kind of $w' \in W$ we are looking for and the reason for that.

\begin{proposition}
Let $C^+_{\mu}:= \{ x \in P(R) \otimes_{\Z} \R \, |\, \langle x, \omega_i \rangle \leq \langle \mu, \omega_i \rangle, \forall i\in I \}$ and suppose that $u \in C^+_{\mu}$. Then $w(u)$ lies in $C^+_{\mu}$ for all $w \in W$ such that $w$ is $u$-minuscule.
\label{prop2}
\end{proposition}

We prove this result below, but first note that $z' \in C^+_{\mu}$. Indeed, recall that $z'=y+ \sum_{j\in J} k^{'}_j \alpha_j.$ All the numbers $k^{'}_i$ belong to the half-open interval $[0,1)$. We would like to prove that $$\langle z' , \omega_{i} \rangle \leq  \langle \mu , \omega_{i} \rangle , \, \forall i \in I.$$ If $ i \in I \setminus J$, then $\langle z', \omega_i \rangle = \langle y, \omega_i \rangle$. But, $\langle y, \omega_i \rangle \leq \langle \mu, \omega_i \rangle, \, \forall i \in I$, since $y$ belongs to the convex hull $\Conv(W\mu)$. Therefore, $\langle z', \omega_i \rangle \leq \langle \mu, \omega_i \rangle, \, \forall i \in I \setminus J.$

If $i \in J$, then $\langle z', \omega_i \rangle = \langle y, \omega_i \rangle + k^{'}_i$. But, $k^{'}_i \in [0,1)$, $\langle y, \omega_i \rangle \leq \langle \mu, \omega_i \rangle$, and, since $y$ and $\mu$ have the same image in $P(R)/Q(R)$, we have $\langle \mu, \omega_i \rangle - \langle y, \omega_i \rangle \in \Z$, thus we get that $$\langle z', \omega_i \rangle \leq \langle \mu, \omega_i \rangle, \, \forall i \in J,$$
and hence $z' \in C^+_{\mu}$. (The last inequalities would not be trivial if we had used $z$ instead of $z'$ because the coefficients $k_i$ may be equal to or bigger than $1$.)

Using Proposition \ref{prop2}, we find that Proposition \ref{prop} follows if we show that there exists an element $w'\in W$ such that $w'(z')$ is dominant and $w'$ is $z'$-minuscule. Before we tackle this problem, let us prove Proposition \ref{prop2}.

\emph{Proof of Proposition \ref{prop2}.}
Suppose that the conditions of the Proposition are satisfied. Let $w\in W$ be $u$-minuscule and suppose that a reduced expression for $w$ is given by $s_{i_1}s_{i_2}\cdots s_{i_t}$. We use induction on $t$, the length of $w$, to prove that $w(u)$ lies in the cone $C^+_{\mu}$, with the case $t=0$ (i.e., $w=id$) already proved. Assume that $s_{i_{r+1}}\cdots s_{i_t}(u)$ lies in $C^+_{\mu}$. We would like to prove that the element $s_{i_r}s_{i_{r+1}}\cdots s_{i_t}(u)$ also lies in $C^+_{\mu}$. Since $\langle s_{i_{r+1}}\cdots s_{i_t}(u) , \alpha_{i_r} \rangle =-1$, we apply the simple reflection $s_{i_r}$ to $s_{i_{r+1}}\cdots s_{i_t}(u)$ to get $s_{i_r}s_{i_{r+1}}\cdots s_{i_t}(u) = s_{i_{r+1}}\cdots s_{i_t}(u) + \alpha_{i_r}$. Then, clearly, for any $i \in I \setminus \{ i_r\}$, we have $\langle s_{i_r}s_{i_{r+1}}\cdots s_{i_t}(u) , \omega_i \rangle = \langle  s_{i_{r+1}}\cdots s_{i_t}(u), \omega_i \rangle \leq \langle \mu, \omega_i \rangle.$ For $i=i_r$ we have that $\langle s_{i_r}s_{i_{r+1}}\cdots s_{i_t}(u) , \omega_{i_r} \rangle = \langle  s_{i_{r+1}}\cdots s_{i_t}(u), \omega_{i_r} \rangle + 1.$ Since $\langle  s_{i_{r+1}}\cdots s_{i_t}(u), \omega_{i_r} \rangle \leq \langle  \mu, \omega_{i_r} \rangle,$ we will be done if we show that we cannot have $\langle  s_{i_{r+1}}\cdots s_{i_t}(u), \omega_{i_r} \rangle = \langle  \mu, \omega_{i_r} \rangle.$

For a contradiction, suppose that $\langle  s_{i_{r+1}}\cdots s_{i_t}(u), \omega_{i_r} \rangle = \langle  \mu, \omega_{i_r} \rangle.$ Then, since $s_{i_{r+1}}\cdots s_{i_t}(u) \in C^+_{\mu}$ and $\mu$ is dominant, there exist non-negative reals $a_i , i \in I \setminus \{ i_r \}$, so that $$s_{i_{r+1}}\cdots s_{i_t}(u) = \mu - \sum_{i \in I \setminus \{ i_r \}}a_i \alpha_i ,$$ and this contradicts our assumption that $\langle s_{i_{r+1}}\cdots s_{i_t}(u), \alpha^{\vee}_{i_r} \rangle = -1$, because $\langle \mu, \alpha^{\vee}_{i_r} \rangle \geq 0,$ $a_i$'s are non-negative, and $\langle \alpha_i, \alpha^{\vee}_{i_r} \rangle \leq 0, \forall i \neq i_r$.
$\square$

Recall that we have reduced the proof of Proposition \ref{prop} to showing that there exists an element $w' \in W$ such that $w'(z')$ is dominant, and $w'$ is $z'$-minuscule. Initially, this problem was proved by the author on a case-by-case basis, but, the following result from \cite{GS} greatly simplifies the proof.

\begin{proposition} \emph{(}\cite{GS}\emph{)}\label{GS}
Let $\lambda \in P(R)$. Then there exists an element $w\in W$ such that $w(\lambda)$ is dominant and $w$ is $\lambda$-minuscule if and only if
\begin{equation}\label{condition}
\langle \lambda, \alpha^{\vee} \rangle \geq -1,
\end{equation}
for all positive coroots $\alpha^{\vee}$ of $R$.
\end{proposition}

In fact, in \cite{GS} a much more general result than Proposition \ref{GS} is proved, but we will only need this special case. The proof of the proposition in this case is fairly elementary (but, for more details, see \cite{GS}). Here we include the proof of the only part of the proposition that we use: that (\ref{condition}) is a sufficient condition for the existence of $w$ as in the proposition. Let $\lambda \in P(R)$ and consider the set $$\mathcal{S}_{\lambda}:=\{ (\alpha, \langle \lambda, \alpha^{\vee}\rangle): \alpha \in R^+, \langle \lambda, \alpha^{\vee}\rangle < 0 \}.$$ Note that if $\langle \lambda , \alpha_i^{\vee} \rangle = -1$, then we have a natural bijection
\begin{equation*}
\mathcal{S}_{\lambda} \setminus \{ (\alpha_i, \langle \lambda, \alpha_i^{\vee} \rangle)\} \longrightarrow \mathcal{S}_{s_i(\lambda)}
\end{equation*}
given by
\begin{equation*}
(\alpha, \langle \lambda, \alpha^{\vee} \rangle ) \longmapsto (s_i(\alpha), \langle \lambda, \alpha^{\vee} \rangle) = (s_i(\alpha), \langle s_i(\lambda), s_i(\alpha^{\vee}) \rangle)  .
\end{equation*}
So, when applying a $\lambda$-minuscule element $w \in W$ to $\lambda$, we get that the size of the set $\mathcal{S}_{\lambda}$ decreases (by an element, for each simple reflection on the reduced expression for $w$). Clearly, this set is finite, therefore we see that there exists an element $w\in W$ such that $w$ is $\lambda$-minuscule and $w(\lambda)$ is dominant.

Using Proposition \ref{GS}, Proposition \ref{prop} now follows from the following result.

\begin{proposition}\label{winning}
For the element $z'$ we have that $\langle z', \alpha^{\vee} \rangle \geq -1$, for all positive coroots $\alpha^{\vee}$ of $R$.
\end{proposition}

The proof of this Proposition is carried out in the next section. In the end, let us mention that one can also phrase propositions \ref{GS} and \ref{winning}, as well as the results of the next section, using a modified version of the numbers game of Mozes (cf. \cite{mozes}), where we impose a lower bound condition (see \cite{GS} for more details).

\section{Proof of Proposition \ref{winning}}

The results in this section are joint with Travis Schedler (stemming from \cite{GS}). We are working under the same assumptions as in the last section. In particular, $R$ is a root system of type ADE. First, we prove that $z$ can be obtained from $z'$ by
applying a $z'$-minuscule Weyl group element to $z'$. More generally, we have:

\begin{lemma}\label{jminlem} Suppose $u \in P(R)$ is minuscule, and
\begin{equation*}
u = \sum_{i \in I} \ell_i \alpha_i,
\end{equation*}
with $\ell_i \geq 0$ for all $i$. Then there exists an element $w\in W$ such that $w$ is $u'$-minuscule and $w(u')=u$, where $u'$ is the \emph{fractional part} of $u$, given by
\begin{equation*}
u' =  \sum_{i \in I} \ell_i' \alpha_i, \quad \ell_i' = \ell_i - \lfloor \ell_i \rfloor.
\end{equation*}
\end{lemma}

In particular, the lemma shows that $z'$ is $J$-minuscule. More generally, it implies that the Weyl orbit of every minuscule configuration
contains its fractional part (which is therefore minuscule). We give a non-case-by-case proof in the next subsection (the only
classification results used are the fact that all (simply-laced) Dynkin diagrams are star-shaped graphs, and that $\widetilde{D_4}$ is extended Dynkin). But, first, we continue with the proof of Proposition \ref{winning}, which will be deduced from the following result.

\begin{proposition} \label{genkrprop}
Let $u \in Q(R_J) \otimes_\Z \R$ be an element such that $\langle u, \alpha^{\vee} \rangle \in [-1,1]$ for all $\alpha \in R_J$. Then, $\langle u, \beta^{\vee} \rangle \in (-2, 2)$ for all $\beta \in R$.
\end{proposition}

To see that this result implies Proposition \ref{winning}, put $u = z'-y$ and note that $u$ satisfies the conditions of the proposition above. (Note, for example, that because $z'$ is $J$-minuscule and $y$ is orthogonal to all the (co)roots of $R_J$, then $\langle u, \alpha^{\vee} \rangle \in [-1,1]$.) Then $$\langle z' , \beta^{\vee} \rangle = \langle u , \beta^{\vee} \rangle + \langle y , \beta^{\vee} \rangle  \geq -1, \forall \beta \in R^+,$$ as desired. (In the last inequality we also used the fact that $y$ is dominant.)

\emph{Proof of Proposition \ref{genkrprop}.} It suffices to show that, if $\beta \in R^+ \setminus R_J$, then $\langle u, \beta^{\vee} \rangle > -2$ for all $u$ as above, where $R^+$ stands for the positive roots in $R$. Fix an element $\beta \in R^+ \setminus R_J$. We claim that the minimum value of $\langle u, \beta^{\vee} \rangle$ is obtained when $u=(-\beta)_J$, the projection of $-\beta$ to $Q(R_J) \otimes_\mathbb{Z} \mathbb{R}$ with respect to the Cartan form.  Then, since $-\beta \notin Q(R_J) \otimes_\Z \R$ (as $\beta \in R \setminus R_J$), it follows that $\langle (-\beta)_J,  \beta^{\vee} \rangle  >  \langle \beta, -\beta^{\vee} \rangle = -2$, as desired.

Let $\mathcal{M}_J = \{v \in Q(R_J) \otimes_\Z \R \mid \langle v, \alpha^{\vee} \rangle \in [-1,1], \forall \alpha \in R_J\}$. Denote by $J_1, \ldots, J_m$ the connected components of $J$. If there exists $j \in J$ such that $\langle \alpha_j, \beta^{\vee} \rangle > 0$, i.e., $\langle \alpha_j, \beta^{\vee} \rangle = 1$, then we may replace $\beta$ with $s_j \beta$ and apply the automorphism $s_j$ to $\mathcal{M}_J$, without changing the statement. Since we can always make an element anti-dominant (in a Dynkin diagram), we may therefore assume that $\langle \alpha_j, \beta^{\vee} \rangle \leq 0$, for all $j \in J$.

For each connected component $J_p$, there exists at most one $j_p \in J_p$ such that $\langle \alpha_{j_p}, \beta^{\vee} \rangle < 0$; moreover, for this $j_p$ we must have that $\langle \alpha_{j_p}, \beta^{\vee} \rangle = -1$. The second statement is clear. For the first one, assume that $\langle \alpha_{j_p}, \beta^{\vee} \rangle < 0$ and $\langle \alpha_{j'_p}, \beta^{\vee} \rangle < 0$, for some $j_p$ and $j'_p \in J_p$, and consider the root that is the sum of the simple roots $\alpha_i$, where $i$ ranges through the set of vertices that, in the Dynkin diagram, form a line segment that starts at $j_p$ and ends at $j'_p$, and that is entirely contained in $J_p$. If $j_p \neq j'_p$, then pairing $\beta^{\vee}$ with this root gives at most $-2$, a contradiction with the fact that we are working with a simply-laced root system.

We may assume that there exists such a $j_p$ in each connected component, since otherwise we could delete the whole connected component from $J$ without changing the statement.

By assumption, $(-\beta)_J \in \mathcal{M}_J$. Therefore, $(-\beta)_J$ attains the minimum value of the function $ \mathcal{M}_J \rightarrow \mathbb{R}$ $( u \mapsto \langle u, \beta^{\vee} \rangle)$ if and only if, for every element $\gamma = \sum_{j \in J} c_j  \alpha_j \in Q(R_J) \otimes_\Z \R$ such that $\langle \gamma , \beta^{\vee} \rangle < 0$, we have that $(-\beta)_J + t \gamma \notin \mathcal{M}_J$ for $t > 0$.

Suppose that $\gamma$ as above satisfies $\langle \gamma , \beta^{\vee} \rangle  < 0$. Then we have $\sum_{p=1}^m c_{j_p} > 0$. Now, pick $p$ such that $c_{j_p} > 0$.  Let $J_p' \subset J_p$ be the maximal connected subset such that $j_p \in J_p'$ and $c_j > 0$ for all $j \in J_p'$. Let $\alpha^{\vee} \in (R_{J_p'})_+$ be the maximal coroot of $J_p'$.  Then, we have $\langle \beta, \alpha^{\vee} \rangle = -1$, so $\langle (-\beta)_J , \alpha^{\vee} \rangle= 1$. Also, we have
\begin{equation*}
\langle \alpha_j, \alpha^\vee \rangle \geq 0, \text{ $ \forall j \in J_p'$ (with strict inequality for at least one $j$)},
\end{equation*}
\begin{equation*}
\langle \alpha_j, \alpha^\vee \rangle \leq 0, \text{ $\forall j \notin J_p'$},
\end{equation*}
\begin{equation*}
c_j > 0, \forall j \in J_p', \, \text{ and}\,\, c_j \leq 0, \forall j \notin J_p' \text{ adjacent to $J_p$}.
\end{equation*}
Thus, we deduce that
\begin{equation*}
\langle \gamma, \alpha^{\vee} \rangle = \sum_{j \in J_p'} c_j \langle \alpha_j, \alpha^\vee \rangle + \sum_{\underset{j \text{ adjacent to } J_p'}{j \in J_p \setminus J_p'}} c_j \langle \alpha_j, \alpha^\vee \rangle > 0.
\end{equation*}
We therefore get $\langle (-\beta)_J + t \gamma, \alpha^{\vee} \rangle = 1 + t (\langle \gamma, \alpha^{\vee} \rangle) > 1$ for all $t > 0$, and hence $(-\beta)_J + t \gamma \notin \mathcal{M}_J$ for any $t > 0$.  Thus, $(-\beta)_J$ indeed attains the minimum value of the function $u \mapsto \langle u, \beta^{\vee} \rangle$ on $\mathcal{M}_J$.
$\square$

Next we prove Lemma \ref{jminlem}.

\emph{Proof of Lemma \ref{jminlem}.} Denote $u_i:=\langle u, \omega_i \rangle$, where recall that $\omega_i$ are the fundamental coweights. Inductively on $\sum_{i \in I} \lfloor \ell_i\rfloor$, it suffices to prove that, if $u \neq u'$, then there exists $i \in I$ such that both $u_i = 1$ and $\ell_i \geq 1$: in this case, we can apply the simple reflection $s_i$ to $u$ and apply the induction hypothesis (for the modified $u$, and the same $u'$).  For a contradiction, suppose that $u \neq u'$, and there does not exist such an $i$.

We claim that not all $\ell_i$ are equal.  If they were equal, then either $\#I = 1$, in which case $u = u'$, a contradiction, or else $\ell_i = 1$ for all $i$ (in order to ensure that $u_i \in \{-1,0,1\}$ at a vertex $i$ of valence $1$). The latter contradicts minusculity, since $\langle u , \delta'^\vee \rangle \geq 2$ where $\delta'^\vee$ is the maximal positive coroot of $R$.

Next, let $j_0 \in I$ be such that $\ell_{j_0}$ is maximal, and such that $\ell_{j_0} > \ell_{i}$ for some $i$ adjacent to $j_0$.  By assumption, $\ell_{j_0} \geq 1$, so we must have $u_{j_0} \in \{0,-1\}$.  If $j_0$ has valence $\leq 2$, then $u_{j_0} = 2\ell_{j_0} - \sum_{i \text{ is adjacent to $j_0$}} \ell_i > 0$, a contradiction.  Hence, the valence of $j_0$ is $3$, and $j_0$ is the node of $\Gamma_R$ (which is a Dynkin diagram of type $D$ or $E$). We will think of $\Gamma_R$ as a star with three branches, each of which contains the node $j_0$.

Let $\Gamma' \subset \Gamma_R$, on the vertex set $I' \subset I$, be the maximal subgraph containing $j$ such that $u_i \in \{0,-1\}$ for all $i \in I'$. (Note that, since $u$ is minuscule, we have that at most one $u_i, i \in I'$ is non-zero.) The restriction $u|_{\Gamma'}$ is antidominant and minuscule on $\Gamma'$. Let $$J := \{i \in I \setminus I' \mid \text{$i$ is adjacent to $\Gamma'$}\}.$$  We must have $u_j = 1$ for all $j \in J$. We claim that $u$ must have the form
\begin{equation*}
u = v + v',
\end{equation*}
where $v$ is supported on one component of $\Gamma_R \setminus \Gamma'$, entirely within one branch of $\Gamma_R$, and $v'$ is supported entirely on a different branch of $\Gamma_R$ (which may or may not intersect $\Gamma'$). If this were not possible, then either $\Gamma_R \setminus \Gamma'$ consists of three components, or else consists of two components and $u$ has an amplitude of $-1$ on a different branch of $\Gamma_R$ from the branches containing $\Gamma_R \setminus \Gamma'$.  The first possibility would contradict minusculity of $u$ because if we let $\alpha^{\vee}$ be the coroot that is the sum of the simple coroots $\alpha_i^\vee,$ where $i$ ranges through the elements of $I'$ and those $i$ that are adjacent to $\Gamma'$, then we would get $\langle u, \alpha^\vee  \rangle \geq 2$. The second possibility would also contradict minusculity of $u$, restricted to the line subsegment of $\Gamma_R$ with endpoints the two vertices of $I \setminus I'$ adjacent to $\Gamma'$ (on this segment, all the $u_i$'s are zero, except at the endpoints, where they are both $1$).

Now, let $j_1 \in J$ be the vertex which lies in the support of $v$.  We will derive a contradiction in the form of $\ell_{j_1} \geq 1$ or $\ell_{j_0} \leq 0$.  Consider the branch of $\Gamma_R$ containing the support of $v$, call it $\Gamma_1 \subset \Gamma_R$, and label its vertices $1, 2, \ldots, n$, in order from the endpoint of the branch to the node $n$ (labeling the vertex $j_0$).  Let $m \in \{1,2,\ldots, n\}$ be the new label of the vertex $j_1 \in J$.  Finally, for any two integers $a \leq b$, let $[a,b] := \{a,a+1,\ldots,b\}$ be the interval of integers between $a$ and $b$, inclusive.

Let us write $v = \sum_{i \in I} f(i) \alpha_i$ and $v' = \sum_{i \in I} g(i) \alpha_i$.  Restricted to the interval $[m,n] \subset \Gamma_1$, both $v$ and $v'$ have at most one nonzero amplitude, which must be on an endpoint of $[m,n]$.  Hence, $f|_{[m,n]}$ and $g|_{[m,n]}$ are linear functions (possibly with constant term).  We can determine exactly what these functions are.  First, $g$ is actually linear restricted to all of $\Gamma_1$, and must be of the form $g(x) = a x$ for some $a \leq 0$, since $\langle v' , \omega_1 \rangle = 0 = 2g(1)-g(2) = g(1) - a$. Next, note that we can find a $v$-minuscule Weyl group element $w$ such that $w(v) = \alpha_c$, for some $c \leq m$, and where the simple reflections appearing in $w$ come from $\Gamma_1 \setminus (\Gamma' \cup \{j_1\}) $.

Suppose that $c = 1$, i.e., $\alpha_c$ is supported on the endpoint of the branch of $\Gamma_R$ containing the support of $v$. In this case, we will derive the contradiction $\ell_{j_1} \geq 1$. Write $\alpha_1 = \sum_{i \in I} q_i \alpha_i$.
Let $h(x)$ be the linear function on $\Z$ such that $h(i) = q_i$ for all $1 \leq i \leq n$. In particular, $h|_{[m,n]} = f|_{[m,n]}$. We have $\ell_{j_1} = f(m) + g(m) = h(m) + g(m)$. We claim that $h(m) + g(m) \geq 1$, which will give the desired contradiction. To see this, first note that $h(x) = bx + 1$, since $2 q_1 - q_2 = 1$.  Next, $\ell_n = h(n) + g(n) = (a+b)n + 1 \geq 1$, so that $a+b \geq 0$.  But then, $h(m) + g(m) = (a+b)m + 1 \geq 1$ as well, giving the desired contradiction.

Finally, suppose that $c \neq 1$. We will derive the contradiction $\ell_{j_0} \leq 0$.  By construction, $\alpha_c + v'$ is in the same Weyl orbit as $u$, and is hence minuscule. Since $c \neq 1$, $c$ cannot be at an extending vertex of $\widetilde{\Gamma_R}$, so $\alpha_c$ is not minuscule. Hence, $v' \neq 0$, and the nonzero amplitude of $v'$ closest to the node must be $-1$.  By applying a $v'$-minuscule element of $W$ to $v$ that comes from the simple reflections on the branch of $\Gamma_R$ containing the support of $v'$, we must obtain a vector $-\alpha_d$, where $d$ is on the same branch of $\Gamma_R$ containing the support of $v'$ (unlike $c$, to the symbol $d$ we do not assign an integer, since $d$ is not necessarily a vertex of $\Gamma_1$). By the same argument as before, we must have that $\alpha_c - \alpha_d$ is minuscule, and moreover, if we write $\alpha_c = \sum_{i \in I} f'(i) \alpha_i$ and $-\alpha_d = \sum_{i \in I} g'(i)\alpha_i$, then $f'|_{[m,n]} = f|_{[m,n]}$ and $g'|_{[m,n]} = g|_{[m,n]}$.  In particular, $f'(n) + g'(n) = \ell_n = \ell_{j_0} \geq 1$.

Since $\alpha_c - \alpha_d$ is minuscule, the restriction of $-\alpha_d$ to the component of $\Gamma_R \setminus \{c\}$ containing $d$, call it $\Gamma_{(d)}$, is minuscule, as is the restriction of $\alpha_c$ to the component of $\Gamma_R \setminus \{d\}$ containing $c$, call it $\Gamma_{(c)}$.  That is, $c$ is an extending vertex of $\Gamma_{(c)}$, and $d$ is an extending vertex of $\Gamma_{(d)}$. Since $c$ is not an endpoint of $\Gamma_{(c)}$ unless $d$ is the node, this can only happen if $d$ is either the node or $\Gamma_{(c)}$ is of type $A$ (a line segment), i.e., $d$ is adjacent to the node.  Then, for $d$ to be an extending vertex of $\widetilde{\Gamma_{(d)}}$, we must have that $\Gamma_{(d)}$ is of type $D_3$ if $d$ is adjacent to the node (since in this case, $\widetilde{\Gamma_{(d)}} \supseteq \widetilde{D_4}$ and hence $\widetilde{\Gamma_{(d)}} = \widetilde{D_4}$), and of type $A$ if $d$ is the node, i.e., $c$ must be adjacent to the node.  In the latter case, $f'(n) + g'(n) < 0$, a contradiction, so we must be in the former case, i.e., $\Gamma_{(d)} \cong D_3$.  In this case, $\Gamma_R$ itself is of type $D_{\geq 4}$, and $c$ is on the long branch.  In this case, it is easy to see that $f'(n) + g'(n) \leq 0$, again a contradiction.
$\square$

\section{The Non-Simply Laced Cases}\label{nonsimplylaced}

In this section we will prove Theorem \ref{thm} for non-simply laced groups. We will use a folding argument to deduce the non-simply laced cases from the simply-laced ones. We thank Robert Kottwitz for generously sharing with us his ideas on proofs of the results in this section.

We retain the same notation as in the Introduction. In particular, $G$ is a split connected reductive group, $B$ is a Borel subgroup, and $T$ is a maximal torus in $B$. We will, furthermore, suppose that $G$ is adjoint and simply-laced. Fix a set of root vectors $\{ X_{\alpha}\}_{\alpha \in \Delta}$ of $T$, where $\Delta$ is the set of simple roots of $G$, with respect to the chosen Borel group $B$.

Let $\theta$ be an automorphism of $G$ that fixes $B$, $T$, and $\{ X_{\alpha}\}_{\alpha \in \Delta}$, and such that the following holds:
\begin{itemize}
\item[($\dagger$)] For every root $\alpha$ from $\Delta$, we have that $\alpha$ is orthogonal to every root $\beta \neq \alpha$ that is in the orbit of $\alpha$ under the group generated by $\theta$, i.e., $( \alpha , \beta ) = 0$, for all $\beta \neq \alpha$ of the form $\beta = \theta ^k (\alpha)$, for some $k \in \mathbb{N}$,
\end{itemize}
where the parentheses $(\, ,)$ stand for the obvious bilinear pairing in $X^*(T) \otimes_{\mathbb{Z}} \mathbb{R}$.

Since $\theta$ acts on $T$, it also acts on the group of characters $X^*(T)$. Denote by $T^{\theta}$ the group of fixed points of $T$ under $\theta$. Then we have that $$X^*(T^{\theta})= X^*(T)_{\theta},$$ where $X^*(T)_{\theta}$ denotes the group of co-invariants of $X^*(T)$ under $\theta$. (In general, for an object on which the map $\theta$ acts, let us agree to use the superscript and subscript $\theta$ for the invariants and co-invariants, respectively, of this object under the action of $\theta$.)

It is clear that $\theta$ acts on $\Delta$. For each orbit of $\theta$ in $\Delta$ we pick a representative, giving us a set which we denote by $\mathcal{R}$ and which we assume is fixed for the rest of this section. The images in $X^*(T)_{\theta}$ of the elements of $\mathcal{R}$ give a basis for $X^*(T)_{\theta}$, and the latter is torsion-free. This means that $X^*(T^{\theta})$ is torsion-free and hence $T^{\theta}$ is connected, which implies that $H:= G^{\theta}$ is also connected. Moreover, $H$ is adjoint since $G$ was assumed to be so. One gets all split adjoint $H$ (up to isomorphism) in this way. We remind the reader (cf. \cite{bourbaki2}, Exercise VII, \S 5, 13, pp. 228--229) that if the Dynkin diagram (or more generally an irreducible component thereof) corresponding to $G$ is of type $A_{2n+1} (n \geq 1)$, $D_{n} (n \geq 4)$, $E_6$, or $D_4$, then the Dynkin diagram (or the respective irreducible component thereof) corresponding to $H$ is of type $B_{n}$, $C_{n-1}$, $F_4$, or $G_2$, respectively, where $\theta$ is of order two in each of the first three cases, apart from the last case where it is of order three.

Recall that by $X$ we have denoted the group of cocharacters $X_*(T)$. We write $Y$ for the group $X^{\theta}$ and note that in fact $Y=X_*(T^{\theta})$. We now consider $$H \supset B^{\theta} \supset T^{\theta},$$ and the Weyl group $W_H$ corresponding to $H$. Since $\Cent_G(A^{\theta})=A$, we get $N_G(A^{\theta}) \subset N_G(A)$, and thus $W_H \leq W$.

In $H$, any Levi component $M_H \supset T^{\theta}$ of a parabolic subgroup containing $B^{\theta}$ arises as the fixed-points group $M^{\theta}$ for some $\theta$-stable Levi component $M \supset T$ of a parabolic subgroup (of $G$) containing $B$. We will write $M_H$ instead of $M^{\theta}$, and remark that it is connected, since $A^{\theta}$ is connected.

We need some more notation. We write $Y_H$ and $Y_{M_H}$ for the quotient of $Y$ by the coroot lattice for $H$ and $M_H$, respectively. The maps $\psi_{H}: Y \rightarrow Y_H $ and $\psi_{M_H}: Y \rightarrow Y_{M_H}$ are the natural projections. We write $\mathfrak{b} = Y \otimes_{\mathbb{Z}} \mathbb{R}$ and $\mathfrak{b}_{M_H} = Y_{M_H} \otimes_{\mathbb{Z}} \mathbb{R}$. The map $pr_{M_H}: \mathfrak{b} \rightarrow \mathfrak{b}_{M_H}$ is the natural projection induced by $\psi_{M_H}$. Finally, for any coweight $\mu \in Y$, $\Conv(W_H (\mu))$ stands for the convex hull in $\mathfrak{b}$ of all the weights in the orbit of $\mu$ under $W_H$.

Let $\mu \in Y$ be $H$-dominant. We define $$\mathcal{P}_{\mu, H} = \left\{ \nu \in Y: (i)\, \psi_H(\nu) = \psi_H(\mu); \, \emph{and} \,\,(ii)\, \nu \in \Conv(W_H(\nu)) \right\}.$$

The following result implies Theorem \ref{thm} for non-simply laced adjoint groups. But, if Theorem \ref{thm} holds for the adjoint group of $G$, then it holds for $G$ itself (see Fact 2, pg. 167, in \cite{cathy}). Therefore the result below implies Theorem \ref{thm} for all non-simply laced $G$, not just the adjoint ones.

\begin{proposition}
With notation as above, we have that $$\psi_{M_H}\left(\mathcal{P}_{\mu, H}\right)=\left\{\nu\in Y_{M_H} :(i)\,\nu, \mu \,\,\emph{have the same image in}\,\, Y_H;\right.$$
$$\left. \hspace{3cm} (ii)\, \emph{the image of}\, \,\nu\, \emph{in} \,\mathfrak{b}_{M_H} \,\emph{lies in}\, pr_{M_H}\left(\Conv \left( W_H(\mu) \right) \right)\right\}.$$
\label{nsl}
\end{proposition}

Before we begin the proof of this proposition, we prove some useful results. First, a remark.

\begin{remark}
\emph{Let us denote by $\mathcal{O}_{\alpha}$ the orbit of $\alpha$ in $\Delta$ under $\theta$. Because of the condition $(\dagger)$ on $\theta$, we have that the coroots corresponding to the simple roots for $(T^{\theta}, H)$ are $N(\alpha^{\vee}) : = \sum_{\gamma \in \mathcal{O}_{\alpha}} \gamma^{\vee},$ where $\alpha$ varies through $\mathcal{R}$. We will need this fact in the proofs of the results that follow. The condition $(\dagger)$ guarantees that our answer is not $2N(\alpha^{\vee})$, which could otherwise happen for certain automorphisms $\theta$.)}
\label{orthogonal}
\end{remark}

\begin{lemma}
Let $\mu \in Y$. Then $\mu$ is $H$-dominant if and only if $\mu$ is $G$-dominant.
\label{hdom}
\end{lemma}

\begin{proof}
The statement of the lemma is a direct consequence of the fact that the simple roots for $(T^{\theta}, H)$ are restrictions to $T^{\theta}$ of the simple roots for $(T, G)$ and the condition that $\mu \in Y$.
\end{proof}

\begin{lemma}
Let $\mu, \nu \in Y$. Then $\nu \stackrel{G}{\leq} \mu \Longleftrightarrow \nu \stackrel{H}{\leq} \mu$.
\label{hleq}
\end{lemma}

\begin{proof}
Recall that $\nu \stackrel{G}{\leq} \mu$, respectively $\nu \stackrel{H}{\leq} \mu$, means precisely that $\mu - \nu$ is a non-negative integral linear combination of the coroots corresponding to the simple roots for $G$, respectively for $H$. We have that
\begin{equation*}
\nu \stackrel{G}{\leq} \mu \Longleftrightarrow \mu - \nu = \sum_{\alpha \in \Delta} c_{\alpha} \alpha^{\vee},
\tag{$\ddagger$}
\end{equation*}
for some $c_{\alpha} \in \mathbb{Z}_{\geq 0}$. Note that since $\mu$ and $\nu$ are fixed by $\theta$, the coefficients $c_{\alpha}$ are constant on the orbits of $\theta$ on the set of the simple roots from $\Delta$. Because of the equivalence $(\ddagger)$, we must have $$\nu \stackrel{G}{\leq} \mu \Longleftrightarrow \mu - \nu = \sum_{\alpha \in \mathcal{R}} d_{\alpha} N(\alpha^{\vee}),$$ for some $d_{\alpha} \in \mathbb{Z}_{\geq 0}$. But, as mentioned in Remark \ref{orthogonal}, the coroots corresponding to the simple roots for $(T^{\theta}, H)$ are $N(\alpha^{\vee})$, where $\alpha$ varies through $\mathcal{R}$. Hence, the last equivalence, according to the definition of $\stackrel{H}{\leq}$, yields $$\nu \stackrel{G}{\leq} \mu \Longleftrightarrow \nu \stackrel{H}{\leq} \mu,$$ which we wanted to prove.
\end{proof}

\begin{lemma}
Let $\mu \in Y$. Denote by $\mathcal{P}(G,\mu)$ the set $\{ \nu \in X: \nu_{\emph{G-dom}} \stackrel{G}{\leq} \mu \}$, where $\nu_{\emph{G-dom}}$ stands for the unique element in $X$  that is in the Weyl orbit $W(\nu)$ and that is $G$-dominant. Similarly, we denote by $\mathcal{P}(H,\mu)$ the set $\{ \nu \in Y: \nu_{\emph{H-dom}} \stackrel{H}{\leq} \mu \}$, where $\nu_{\emph{H-dom}}$ stands for the unique element in $Y$  that is in the Weyl orbit $W_H(\nu)$ and that is $H$-dominant. Then we have that $$\mathcal{P}(H, \mu) = Y \cap \mathcal{P}(G, \mu).$$
\label{plemma}
\end{lemma}

\begin{proof}
Since $\mu$ is in $Y$, the result is immediate from Lemmas \ref{hdom} and \ref{hleq}.
\end{proof}

\begin{lemma}
We have the following commutative diagram where the vertical maps are the obvious projections
$$ \begin{tabular}{ccc}
$Y$ & $\subset$ & $X$ \\
$\downarrow$ & & $\downarrow$ \\
$Y_{M_H}$ & $\hookrightarrow$ & $X_M$ \\
$\downarrow$ & & $\downarrow$ \\
$Y_H$ & $\hookrightarrow$ & $X_G$. \\
\end{tabular} $$
\label{diagram}
\end{lemma}

\begin{proof}
We only need to explain why the horizontal maps are (natural) inclusions. This is clear for the first map. For the third map, recall from Remark \ref{orthogonal} that the coroots corresponding to the simple roots for $(T^{\theta}, H)$ are $N(\alpha^{\vee})$, where $\alpha$ varies through $\mathcal{R}$. This implies that the coroot lattice for $H$ is the intersection of $Y$ with the coroot lattice for $G$, and thus the third map is an inclusion.

Now we will prove that the second map is also an inclusion, with the proof being almost identical to that of the similar fact for the third map. Similar to Remark \ref{orthogonal}, because of condition $(\dagger)$, we have that the coroot lattice for $T^{\theta}$ in $M_H=M^{\theta}$ has a $\mathbb{Z}$-basis consisting of $N(\alpha^{\vee})$, where $\alpha$ varies through a set of representatives for orbits of $\theta$ on $\Delta_M$, and where $\Delta_M$ is the set of the simple roots for $M$. This implies that the coroot lattice for $T^{\theta}$ in $M_H=M^{\theta}$ is just the intersection of $Y$ with the coroot lattice for $T$ in $M$. This ensures that the second horizontal map is injective. That the diagram is commutative follows directly from the definitions of the maps involved.
\end{proof}

We now start the proof of Proposition \ref{nsl}.

\begin{proof}
It is clear that the left-hand side is contained in the right-hand side. The point is to show that the converse is true as well. Let $\nu \in Y_{M_H}$ be an element of the set appearing on the right-hand side in Proposition \ref{nsl}. We may assume that $\nu$ is $H$-dominant in $\mathfrak{b}_{M_H}$ (otherwise we could pick some other Borel $B^{\theta}$ in $H$ with respect to which $\nu$ is $H$-dominant). Thus we have the following important properties for $\nu \in Y_{M_H}$: $\nu$ is $H$-dominant, $\nu \stackrel{H}{\leq} \mu$, and $\nu$ and $\mu$ have the same image in $Y_H$.

Using lemmas \ref{hdom}, \ref{hleq}, and \ref{diagram} we see that: $\nu$ is $G$-dominant, $\nu \stackrel{G}{\leq} \mu$, and $\nu$ and $\mu$ have the same image in $X_G$. (Using Lemma \ref{diagram}, we are viewing $\nu$ as an element in $X_M$.) Let $\tilde{\nu} \in X$ be the unique $M$-dominant, $M$-minuscule representative of $\nu$. The results of Section 1 guarantee that $\tilde{\nu} \in \mathcal{P}(G,\mu)$. Then $\theta (\tilde{\nu})$ is the unique $M$-dominant, $M$-minuscule representative of $\theta(\nu)=\nu$. So $\theta(\tilde{\nu})=\tilde{\nu}$, in other words $\tilde{\nu} \in Y$. Using Lemma \ref{plemma} we see that, since $\tilde{\nu}$ lies in both $Y$ and $\mathcal{P}(G, \mu)$, it also lies in $\mathcal{P}(H, \mu)$. We already know that $\nu$ and $\mu$ have the same image in $Y_H$, and since $\tilde{\nu}$ evidently maps to $\nu$, we have that $\nu$ is an element of $\psi_{M_H}\left(\mathcal{P}_{\mu, H}\right)$, thus concluding the proof of our proposition.
\end{proof}

\section{The case of quasi-split groups}

We now work with groups that are quasi-split. Let us fix the notation, since it is slightly different from that introduced in the Introduction. Let $F$ be a finite extension of $\mathbb{Q}_p$ with uniformizing element $\pi$, and let $L$ be the completion of the maximal unramified extension of $F$ in some algebraic closure of $F$. Denote by $\mathfrak{o}_F$, resp. $\mathfrak{o}_L$, the ring of integers in $F$, resp. $L$, and by $\sigma$ the Frobenius automorphism of $L$ over $F$. Let $G$ be a connected reductive group that is quasi-split over $F$ and split over $L$. Let $A$ be a maximal split torus in $G$, and $T$ its centralizer. Let $B=TU$ be a Borel subgroup of $G$, containing $T$ and $U$ the unipotent radical of $B$. Let $P=MN$ be a parabolic subgroup containing $B$, with $M \supset T$ and $N$ the unipotent radical of $P$. Suppose that all of the above groups are defined over $\mathfrak{o}_F$.

The definition of affine Deligne-Lusztig varieties remains the same as in the split case: $$X_{\mu}^G(b):= \{ x \in G(L)/G(\mathfrak{o}_L) : \, x^{-1} b \sigma (x) \in G(\mathfrak{o}_L) \mu (\pi) G(\mathfrak{o}_L) \},$$ with $\mu \in X_*(T)$ dominant and $b \in M(L)$.

Let $X_M$ denote the quotient of the cocharacter lattice $X_*(T)$ of $T$ by the coroot lattice for $M$. The Frobenius automorphism $\sigma$ acts on $X_M$, and we denote by $Y_M$ the cooinvariants of this action, i.e., $Y_M:= X_M/(1-\sigma)X_M$. Write $Y$ for the coinvariants of $X_*(T)$, and note that we have the following commutative diagram
$$\begin{array}{ccc}
  X_*(T) & \rightarrow & X_M \\
  \downarrow &  &  \downarrow \\
  Y & \rightarrow & Y_M
\end{array}$$
where all the maps are surjective. We denote the map $X_*(T) \rightarrow Y$ by $\rho$. We write $\psi$ for the map $Y \rightarrow Y_M$ from the above diagram, and then write $\phi: X_*(T) \twoheadrightarrow Y_M$ for the composition $\psi \circ \rho$.

Denote by $\stackrel{P}{\preceq}$ the partial ordering on $Y_M$ defined as follows: For $y_1, y_2 \in Y_M$, we write $y_1 \stackrel{P}{\preceq} y_2$ if $y_2-y_1$ is a nonnegative integral linear combination of the images in $Y_M$ of the coroots $\{ \alpha_j^{\vee}: j \in J \}$ corresponding to simple roots $\{ \alpha_j: j \in J \}$ of $T$ in $N$.

Similarly to the Kottwitz maps in the split case from the Introduction, we again have such maps in the quasi-split case, $w_G: G(L) \rightarrow X_G$ and $w_M: M(L) \rightarrow X_M$. The latter induces a map $\kappa_M:B(M) \rightarrow Y_M$ (see \cite{kotiso} for the precise definition), where $B(M)$ stands for the $\sigma$-conjugacy classes in $M(L)$.

Similar to Theorem \ref{groupmazur} in the case of split groups, we have the following result:

\begin{theorem}  Let $\mu \in X_*(T)$ be dominant and let $b \in M(L)$ be a basic element such that $\kappa_M(b)$ lies in $Y_M^{+}$. Then $X_{\mu}^G(b)$ is non-empty if and only if $\kappa_M(b) \stackrel{P}{\preceq} \mu$.
\label{mainquasi}
\end{theorem}

One implication, namely that $X_{\mu}^G(b)$ being non-empty implies $\kappa_M(b) \stackrel{P}{\preceq} \mu$, is the group-theoretic version of Mazur's Inequality and a proof of this fact can be found in \cite{kot}, Theorem 1.1, part (1). For the converse, Kottwitz and Rapoport (cf. \cite{kot}, \S4.3) showed that it follows from Theorem \ref{conjecture2} below, which they conjectured to be true. To state their conjecture, we need some more notation.

We fix a dominant element $\mu \in X_*(T)$, and, as in the Introduction, we define the set $\mathcal{P}_{\mu}: = \{ \nu \in X_*(T): \nu = \mu$ in $X_G, \, \nu \in \Conv(W\mu)\}$, where $X_G$ is the quotient of $X_*(T)$ by the coroot lattice for $G$, and $\Conv(W\mu)$ stands for the convex hull in $X_*(T)\otimes_{\mathbb{Z}}\mathbb{R}$ of the Weyl group orbit of $\mu$. Write $\mathcal{P}_{\mu, M}$ for the image of $\mathcal{P}_{\mu}$ under the map $\phi: X \twoheadrightarrow Y_M$.

Let $A_P$ be the maximal split torus in the center of $M$ and let $\mathfrak{a}_P:=X_*(A_P)\otimes_{\mathbb{Z}}\mathbb{R}$, where the last space is viewed as a subspace of $X_*(T)\otimes_{\mathbb{Z}}\mathbb{R}$. Identifying $Y_M\otimes_{\mathbb{Z}}\mathbb{R}$ with $\mathfrak{a}_P$, we write $Y_M^+$ for the subset of $Y_M$ consisting of elements whose images in $\mathfrak{a}_P$ lie in the set
\begin{equation*}
\{ x\in \mathfrak{a}_P : \langle \alpha, x \rangle > 0, \, \text{for all roots} \,\, \alpha \,\, \text{of} \,\, A_P\,\, \text{in}\,\, N \}.
\end{equation*}

Theorem \ref{mainquasi} follows from the following

\begin{theorem}\label{conjecture2} \emph{(Kottwitz-Rapoport Conjecture; quasi-split case)} Let $\mu \in X_*(T)$ be dominant and $\nu_M \in Y_M^+$. The following are equivalent:
\begin{itemize}
\item[(i)] $\nu_M \stackrel{P}{\preceq} \mu$
\item[(ii)] $\nu_M \in \mathcal{P}_{\mu, M}$.
\end{itemize}
\end{theorem}

(In the condition (i) above we consider $\mu$ as an element of $Y_M$.) One sees immediately that (ii) implies (i). The point is to prove that (i) implies (ii). We give a proof of this implication below.

In her Ph.D. thesis \cite{cathy2}, Lucarelli proved Theorem \ref{conjecture2} for unitary groups of rank 3, 4, and 5. Many of her arguments are general and apply to other groups, however, so we will use her ideas and exposition. The crucial added ingredient here is the use of a lemma of Stembridge and of Proposition \ref{GS}.

\subsection{Proof of Theorem \ref{conjecture2}}

Suppose that $\nu_M \in Y_M^+$ and that $\nu_M \stackrel{P}{\preceq} \mu$, where $\mu \in X_*(T)$ is dominant. We would like to prove that there exists an element $\nu \in \mathcal{P}_{\mu}$ such that $\nu \mapsto \nu_M$ under the map $\phi: X \twoheadrightarrow Y_M$. For this purpose, define $$\mathcal{P}_{\rho(\mu)}': = \{ y \in Y: (i) \, \, y \,\, \text{and} \,\, \rho(\mu)\,\, \text{have same image in}\,\, Y_G, (ii) \,\, y\in \Conv(W'\rho(\mu))\},$$ where $W'$ is the Weyl group associated with $Y$, i.e., the relative Weyl group $N(A)(F)/T(F)$, and $\Conv(W'\rho(\mu))$ is the convex hull in $Y\otimes_{\mathbb{Z}}\mathbb{R}$ of the orbit $W'\rho(\mu)$. Then, since we know that Theorem \ref{conjecture2} is true for all the corresponding root systems of reduced and non-reduced type (see Remark \ref{remark} below), we can find $\rho(\nu) \in \mathcal{P}_{\rho(\mu)}'$ such that $\psi(\rho(\nu))=\nu_M$. Thus it is sufficient to prove that the image of $\mathcal{P}_{\mu}$ under the map $\rho: X_*(T) \rightarrow Y$ equals $\mathcal{P}_{\rho(\mu)}'$.

\begin{remark}\label{remark}
\emph{In the split case, since we did not need it there, we did not consider the root system of type $BC_n$, the only non-reduced irreducible root system. However, one can deduce Theorem \ref{conjecture} for $BC_n$ through the process of folding (the root system $A_{2n}$), where one no longer assumes condition (\dag) from the previous section.}
\end{remark}

We first show that $\rho (\mathcal{P}_{\mu}) \subset \mathcal{P}_{\rho(\mu)}'$. Suppose that $x \in \mathcal{P}_{\mu}$. Then $x$ has the same image in $X_G$ as $\mu$, under the canonical map $X_*(T) \twoheadrightarrow X_G$. Hence $\rho(x)$ and $\rho(\mu)$ have the same image in $Y_G$. So, it suffices to prove that $x \in \Conv(W\mu)$ implies $\rho(x) \in \Conv(W'\rho(\mu))$. For this, we will use two easy facts (whose proofs are omitted):

\begin{itemize}
\item[(a)] If $x$ is dominant for $X_*(T)$, then $\rho(x)$ is dominant for $Y$, and
\item[(b)] If $x \stackrel{!}{\geq} 0$ for $X_*(T)$, then $\rho(x) \stackrel{P}{\succeq} 0$ for $Y$.
\end{itemize}
(Here $\stackrel{!}{\geq}$ denotes the usual partial ordering in $X_*(T)$, where $x_1 \stackrel{!}{\geq} x_2 $ means that $x_1-x_2$ is a nonnegative integer linear combination of simple coroots of $T$ in $N$.)

From $x\in \Conv(W\mu)$ and $\mu$ being dominant, we get that $wx \stackrel{!}{\leq} \mu$ for all $w \in W$, and thus $w'x \stackrel{!}{\leq} \mu$ for all $w' \in W'$, since we can regard $W'$ as a subgroup of $W$. Using (a) and (b) we then get that $\rho(\mu)$ is dominant and that $\rho(w'x) \stackrel{P}{\preceq} \rho(\mu)$ for all $w' \in W'$. But the action of $W'$ commutes with $\rho$, so we have $w'\rho(x) \stackrel{P}{\preceq} \rho(\mu)$ for all $w' \in W'$, and thus $\rho(x) \in \Conv(W'\rho(\mu))$, since $\rho(\mu)$ is dominant. This shows that $x \in \Conv(W\mu)$ implies $\rho(x) \in \Conv(W'\rho(\mu))$, as desired.

Now that we know that $\rho (\mathcal{P}_{\mu}) \subset \mathcal{P}_{\rho(\mu)}'$, we would like to prove the other inclusion. Suppose that $\nu' \in \mathcal{P}_{\rho(\mu)}'$. Without loss of generality we may assume that $\nu'$ is dominant. Then we have that $\nu' \stackrel{P}{\preceq} \rho(\mu)$. Due to the transitive property of $\stackrel{P}{\preceq}$, we only need to consider the case when $\rho(\mu)$ covers $\nu'$. Recall that we say that $\rho(\mu)$ \emph{covers} $\nu'$ if for any $\upsilon$ with $\nu' \stackrel{P}{\preceq} \upsilon \stackrel{P}{\preceq} \rho(\mu)$ we have that $\upsilon = \nu'$ or $\upsilon = \rho(\mu)$. Suppose, therefore, that $\rho(\mu)$ covers $\nu'$. Then using a lemma of Stembridge (\cite{stem2}, Cor. 2.7; see also \cite{rapoport}, Lemma 2.3, for an alternative proof, due to Waldspurger, of this result) we can conclude that there exists a positive coroot $\beta^{\vee}$ such that $\nu'= \rho(\mu) - \beta^{\vee}$. Thus, in order to prove Theorem \ref{conjecture2}, it suffices to show the following:

\begin{proposition} \label{problem}
There exists an element $\nu \in \mathcal{P}_{\mu}$ such that $\rho(\nu) = \rho(\mu)-\beta^{\vee}$.
\end{proposition}

\begin{proof} Denote by $R$ the root system formed by the coroots of $G$ in $X_*(T)$, and by $R'$ the one obtained by taking the image of $R$ under $\rho$. More precisely, the coroots of $R'$ are obtained by taking the images of the coroots of $R$ under the map $\rho$.

\begin{remark}
\emph{Note that we do not get $R'$ from $R$ by \emph{folding}, which would amount to taking invariants under the automorphism $\sigma$. Rather, we are \emph{cofolding} the root system $R$ to get $R'$, i.e., we are taking the coinvariants under the action of $\sigma$. For example, folding $A_{2n-1}$ would yield $B_n$, but cofolding $A_{2n-1}$ yields $C_n$. Also, we remark again that we are working with coroots and not roots.}
\end{remark}

One sees immediately that there exists a (positive) coroot $\gamma^{\vee}$ such that $\rho (\gamma^{\vee}) = \beta^{\vee}$  and $\langle \mu, \gamma \rangle \geq 1.$ Indeed, since $\mu$ is dominant, we have that for all coroots $\gamma^{\vee}$ with the property that $\rho (\gamma^{\vee}) = \beta^{\vee}$, we must have that $\langle \mu, \gamma \rangle \geq 0$. If for all these $\gamma$ we had $\langle \mu, \gamma\rangle = 0$, then we would get $\langle \rho(\mu), \beta \rangle = 0$. But this would give $\langle \nu', \beta \rangle = -2,$ contradicting the assumed dominance of $\nu'$.

Now we put  $\nu: = \mu - \gamma^{\vee}$. Denote by $\varpi_j$, $j \in J$, the set of fundamental weights, where $\varpi_j$ corresponds to the simple coroot $\alpha_j^{\vee}$, in $X_*(T)$. Recall the definition of the cone $$C^+_{\mu}:= \{ u \in X_*(T)\otimes_{\mathbb{Z}}\mathbb{R} \, |\, \langle u, \varpi_j \rangle  \leq \langle \mu, \varpi_j \rangle  \}$$ and see immediately that $\nu \in C^+_{\mu}$. Using Proposition \ref{prop2}, we see that Proposition \ref{problem} follows if we show that there exists an element $w\in W$ such that $w$ is $\nu$-minuscule and $w(\nu)$ is dominant. But, from Proposition \ref{GS} we have that this is the case if and only if
\begin{equation*}
\langle \nu, \alpha \rangle \geq -1, \,\, \text{for all $\alpha \in R^+$}.
\end{equation*}
It remains to prove that these inequalities hold in our case.

Recall that $\nu = \mu - \gamma^{\vee}$, where $\gamma^{\vee}$ is a positive coroot such that $\langle \mu, \gamma \rangle \geq 1.$ If $\alpha \in R^+ \setminus \{\gamma\}$, since $R$ is simply-laced, it is well known that we must have $\langle \gamma^{\vee},\alpha \rangle \leq 1$. Therefore $\langle \nu, \alpha \rangle = \langle \mu, \alpha \rangle - \langle \gamma^{\vee}, \alpha \rangle \geq 0 - 1 =-1.$ For $\alpha = \gamma$, we have $\langle \nu , \gamma \rangle = \langle \mu , \gamma \rangle - \langle \gamma^{\vee}, \gamma \rangle \geq 1 -2 = -1.$ This concludes the proof of Proposition \ref{problem} and therefore of Theorem \ref{conjecture2}.
\end{proof}

\def\refname{R\MakeLowercase{eferences}}


\begin{thebibliography}{99}
\bibitem{bourbaki} N. BOURBAKI, \textsl{Elements of Mathematics, Lie groups and Lie algebras, Chapters 4-6}, Springer-Verlag, Berlin Heidelberg 2002.
\bibitem{bourbaki2} N. BOURBAKI, \textsl{\'El\'ements de Math\'ematique, Groups et alg\'ebres de Lie, Chapitres 7 et 8}, Springer-Verlag, Berlin Heidelberg 2006.
\bibitem{fulton} W. FULTON, \textsl{Introduction to Toric Varieties}, Ann. of Math. Stud. 131, Princeton Univ. Press, Princeton, NJ, 1993.
\bibitem{qendrim} Q. R. GASHI, \textsl{Vanishing results for toric varieties associated to $GL_n$ and $G_2$}, Transform. Groups vol. 13, no. 1 \textbf{(2008)}, 149--171.
\bibitem{qendrim2} Q. R. GASHI, \textsl{A vanishing result for toric varieties associated with root systems}, Albanian J. Math. (2007), \textbf{4}, 235--244.
\bibitem{gashischedler} Q. R. GASHI and T. SCHEDLER, \textsl{On dominance and minuscule Weyl group elements}, preprint in preparation.
\bibitem{GHKR} M. G\"ORTZ, T. J. HAINES, R. KOTTWITZ and D. C. REUMAN, \textsl{Dimensions of some affine Deligne-Lusztig varieties}, Ann. sci. \'Ecole Norm. Sup. $4^e$ s\'erie, (2006), \textbf{39}, 467-511.
\bibitem{GHKR2} M. G\"ORTZ, T. J. HAINES, R. KOTTWITZ and D. C. REUMAN, \textsl{Affine {D}eligne-{L}usztig varieties in affine flag varieties}, arXiv:0805.0045v2, 2008.
\bibitem{kotiso} R. E. KOTTWITZ, \textsl{Isocrystals with additional structure},  Compositio Math. \textbf{56,} (1985), no.2, 201--220.
\bibitem{kot} R. E. KOTTWITZ, \textsl{On the Hodge-Newton decomposition for split groups},  Int. Math. Res. Not. \textbf{2003,} no.26, 1433--1447.
\bibitem{akot} R. E. KOTTWITZ, \textsl{Harmonic analysis on reductive $p$-adic groups and Lie algebras}, in \textsl{Harmonic Analysis, the Trace Formula, and Shimura Varieties}, 393--522, Clay Math. Proc., \textbf{4,} Amer. Math. Soc., Providence, RI, 2005.
\bibitem{krapo} R. E. KOTTWITZ and M. RAPOPORT, \textsl{On the existence of F-isocrystals}, Comment. Math. Helv. \textbf{78} (2003), 153--184.
\bibitem{cathy} C. LUCARELLI, \textsl{A converse to Mazur's inequality for split classical groups}, Journal of the Inst. Math. Jussieu (2004) \textbf{3} (2), 165--183.
\bibitem{cathy2} C. LUCARELLI, \textsl{A converse to Mazur's inequality for split classical groups}, Ph.D. Thesis, The University of Chicago, June 2004.
\bibitem{mazur} B. MAZUR, \textsl{Frobenius and the Hodge filtration}, Bull. Amer. Math. Soc. \textbf{78} (1972), 653--667.
\bibitem{mazur2} B. MAZUR, \textsl{Frobenius and the Hodge filtration (estimates)}, Ann. of Math. (2) \textbf{98} (1973), 58--95.
\bibitem{mozes} S. MOZES, \textsl{Reflection processes on graphs and Weyl groups} J. Combin. Theory Ser. A \textbf{53} (1990), no.1, 128--142.
\bibitem{rapoport} M. RAPOPORT, \textsl{A positivity property of the Satake isomorphism} Manuscripta Math. \textbf{101} (2000), no.2, 153--166.
\bibitem{rapoport2} M. RAPOPORT, \textsl{A guide to the reduction modulo $p$ of Shimura varieties}, in \textsl{Automorphic forms (I)}, Ast\'erisque No. \textbf{298} (2005), 271--318.
\bibitem{rapritch} M. RAPOPORT and M. RICHARTZ, \textsl{On the classification and specialization of F-isocrystals with additional structure}, Compositio Math. \textbf{103} (1996), no.2, 153--181.
\bibitem{stem2} J. R. STEMBRIDGE, \textsl{The partial order of dominant weights}, Adv. Math. \textbf{136} (1998), 340--364.
\bibitem{stem} J. R. STEMBRIDGE, \textsl{Minuscule elements of Weyl groups}, J. Algebra \textbf{235} (2001), no.2, 722--743.
\bibitem{eva} E. VIEHMANN, \textsl{The dimensions of affine Deligne-Lusztig varieties}, Ann. sci. \'{E}cole Norm. Sup. $4^e$ s\'erie \textbf{39} (2006), 513--526.
\bibitem{win} J.-P. WINTENBERGER, \textsl{Existence de $F$-cristaux avec structures suppl\'ementaires},  Adv. Math. \text{190} (2005), no. 1, 196--224.

\end{thebibliography}
\end{document}